\documentclass[12pt,reqno]{amsart}
\usepackage{amsmath}  
\usepackage{amsmath,amssymb}  
\usepackage[margin=1in]{geometry}
\usepackage{epstopdf}
\usepackage{graphicx}
\usepackage{tikz}
\usepackage{comment}
\usepackage[breaklinks=True]{hyperref}
\usepackage{multirow}
\usepackage{makecell}
\usepackage{longtable}
\usepackage{fancyhdr}

\usetikzlibrary{arrows}

\newtheorem{theorem}{Theorem}[section]
\newtheorem{example} [theorem]{Example}
\newtheorem{remark}[theorem]{Remark}
\newtheorem{proposition}[theorem]{Proposition}

\newcommand{\tabincell}[2]{\begin{tabular}{@{}#1@{}}#2\end{tabular}}

\def\R{\Bbb R}  \def\Z{\Bbb Z}       \def\CC{\Bbb C}
\def\={\;=\;}  \def\:{\;:=\;}   \def\+{\,+\,}

\def\be{\begin{equation}}  \def\ee{\end{equation}}  \def\bes{\begin{equation*}}  \def\ees{\end{equation*}}
\def\ba{\be\begin{aligned}}  \def\ea{\end{aligned}\ee}  \def\bas{\bes\begin{aligned}}  \def\eas{\end{aligned}\ees}

\numberwithin{equation}{section}

\begin{document}

\title[Some new Ramanujan-Sato series for 1/$\pi$]{Some new Ramanujan-Sato series for 1/$\pi$}
\author{$\ ^{1}$Tao Wei}  \author{$\ ^{2}$Zhengyu Tao} \author{$\ ^{3}$Xuejun GUO}
\dedicatory{$\ ^{1,2,3}$Department of Mathematics, Nanjing University\\
	Nanjing 210093, China\\
	$\ ^{1}$weitao@smail.nju.edu.cn  $\ ^{2}$taozhy@smail.nju.edu.cn $\ ^{3}$guoxj@nju.edu.cn}
\thanks{The authors are supported by  NSFC (Nos.11971226, 11631009).}
\maketitle

\begin{abstract} We derive 10 new Ramanujan-Sato series of $1/\pi$ by using the method of Huber, Schultz and Ye. The levels of these series are $14,15,16,20,21,22,26,35,39$.
\end{abstract}

\medskip

\textbf{Keywords:} Hauptmodul; modular polynomial; Moonshine group.

\section{Introduction}
\label{chap01}

In  Ramanujan's paper \cite{RAMA}, he announced 17 series for $1/\pi$, those series have the form of 
\begin{equation}\label{Ramanujan-Sato}
	\frac{1}{\pi}=\sum_{n=0}^{\infty}A_n(Bn+C)X^n.
\end{equation}
For example, one of his series is
\[\frac{1}{\pi}=\frac{2\sqrt{2}}{9801}\sum_{n=0}^\infty\frac{(4n)!}{(n!)^44^{4n}}\frac{1103+26390n}{99^{4n}}.\]
It was not until 1987 in \cite{BOR} that all of his formulas were proved. Nowadays, series of the form (\ref{Ramanujan-Sato}) is called the Ramanujan-Sato series due to Sato's research on this type of series. Moreover, Borweins and Chudnovskys derived some new $1/\pi$ series similar to Ramanujan's independently in \cite{CHUD}.

More recently, Chan, Chan and Liu provided a systematic classification of these series by the levels of the modular forms in \cite{CCL}, see \cite{COOP} for a summary with levels less than 12.
 In an unpublished preprint \cite{TDDN}, Huber, Schultz and Ye constructed a systematically method to find new families of Ramanujan-Sato series by using the Hauptmodul for some Moonshine groups. They also derived series of levels $17$ and $20$ in \cite{TDD1}, \cite{TDD2}.
 
We briefly introduce Huber et al.'s method. Let $\Gamma$ be a genus zero discrete subgroup of $\mathrm{SL}_2(\R)$ with Hauptmodul $x(\tau)$ (see Section \ref{HaDE} for definitions). There is a weight $2$ modular form $z(\tau)$ that satisfies a third-order differential equation
\[
	2wz_{xxx}+3w_xz_{xx}+(w_{xx}-2R)z_x-R_xz=0,
\]
where $w,R$ are polynomials in $x$. From this equation, we can express $z$ as
\[z=\sum\limits_{n=0}^\infty A_nx^n.\]
By picking some special CM-points $\tau_0\in\mathrm{H}$ and using the modular equations satisfies by $x$, we can derive series of the form (\ref{Ramanujan-Sato}).

In \cite{TDDN}, $x,w,R$ that associated with some moonshine groups are listed in several huge tables, the Hauptmoduls that they list all have the Fourier expansion $x=q+O(q^2)$. If $x$ has Fourier expansion $x=1/q+O(1)$, then we can derive the modular equation. It turns out that we can also find the modular equations for $x=q+O(q^2)$ in several cases. In this paper, we search
 the Hauptmoduls that has a modular equation and derive 10 new Ramanujan-Sato series.

\section{Hauptmodul and Differential equations}\label{HaDE}

Let $\Gamma$ be a genus zero discrete subgroup of $\mathrm{SL}_2(\mathbb{R})$ commensurable with $\mathrm{SL}_2(\mathbb{Z})$, i.e. 
\[[\Gamma:\Gamma\cap \mathrm{SL}_2(\Z)],\  [\mathrm{SL}_2(\Z):\Gamma\cap \mathrm{SL}_2(\Z)]<\infty\]
and the modular curve $\Gamma\setminus\mathcal{H}^*$ is a compact Riemann surface of genus zero, where $\mathcal{H}^*$ is the upper half plane together with cusps. From the knowledge of Riemann surfaces, we know that the field of meromorphic functions on $\Gamma\setminus\mathcal{H}^*$ is generated by a single element $t_{\Gamma}$ that transcendental over $\CC$, such $t_{\Gamma}$ is called  a Hauptmodul for $\Gamma$.

For two natural numbers $N$ and $e||N,$ i.e. $e|N$ and $\mathrm{gcd}(e,N/e)=1$, let
\[	W_e=\left\{\begin{pmatrix}
	ea & b \\
	Nc & ed
\end{pmatrix}\Big| (a,b,c,d)\in \mathbb{Z}^4,\text{ } ead-\frac{N}{e}bc=1\right\}\]
be the set of Atkin-Lehner involutions. For any set of indices $e$ closed under the following law of multiplication 
\[W_eW_f\equiv W_{ef/\text{gcd}(e,f)^2}\text{ mod }\Gamma_0(N),\]
the Moonshine group $\Gamma:=\cup_e W_e$ is a subgroup of the normalizer of $\Gamma_0(N)$. We denote such a $\Gamma$ by $N+e_1,e_2,\cdots$, or $N+$ if all of the indices are present. More details about Moonshine groups can be found in Conway's Monstrous Moonshine paper \cite{CONO}.

From now on, let's focus on the genus zero Moonshine group $\Gamma$. Since $\begin{pmatrix}
	1 & 1 \\
	0 & 1
\end{pmatrix}\in\Gamma$, $t_\Gamma$ has a Fourier expansion $\displaystyle t_\Gamma=\sum_{n=m}^{\infty}a_nq^n$ for some $m\in\Z$, where $q=e^{2\pi\mathrm{i}\tau}$. Notice that for any Hauptmodul $\displaystyle t_\Gamma$ and $\displaystyle \begin{pmatrix}
a & b \\
c & d
\end{pmatrix}\in\mathrm{SL}_2(\Z),$ the Möbius transformation $\displaystyle \frac{at_\Gamma+b}{ct_\Gamma+d}$ is also a Hauptmodul for $\Gamma$. The Hauptmodul $t_\Gamma$ is said to be normalized if it has the Fourier expansion of the form 
\[t_\Gamma=\frac{1}{q}+\sum_{n=1}^{\infty}a_nq^n.\]

Let $f(x)$ be a function, here and throughout the paper, denote $\displaystyle f_x=x\frac{df}{dx}$ be the same as \cite{TDDN}.
\begin{theorem}[Theorem 2.2 and 2.3 of \cite{TDDN}]\label{MDFz}
	For any choice of $x(\tau)$ as a Möbius transformation of Hauptmodul $t_\Gamma$, there exists a polynomial $w(x)$ and a modular form $\displaystyle z=\frac{(\log(x))_q}{\sqrt{w(x)}}$ of weight $2$, so that $\displaystyle R:=\frac{2zz_{qq}-3z_q^2}{z^4}$ is a polynomial in $x$, and we have a differential equation for $z$ with respect to $x$:
	\begin{equation}\label{DE}
		2wz_{xxx}+3w_xz_{xx}+(w_{xx}-2R)z_x-R_xz=0.
	\end{equation}
\end{theorem}

By this theorem, we can write $z$ as a power series in $x$, i.e. $\displaystyle z=\sum\limits_{n=0}^\infty A_nx^n$ and the coefficients $A_n$ can be obtained from the recurrence relation implied by the differential equation. In \cite{TDDN}, there is a table that lists some Moonshine groups together with their Hauptmoduls. For every $\Gamma$ and some choice of $x(\tau)$ as a Möbius transformation of Hauptmodul, the authors also calculated the $w(x)$ and $R(x)$ associated with $x(\tau)$. We give here an example of the calculation of the recurrence relation for $\Gamma=39+39$. Before this example, recall that Dedekind eta-function $\displaystyle \eta(\tau)=e^{2\pi\mathrm{i}\tau/24}\prod_{j=1}^{\infty}(1-q^j)$, we define $\eta_l(\tau):=\eta(l\tau)$.

\begin{example}\label{Anfor39+39}
	If $\Gamma=39+39$, then $\displaystyle t_{\Gamma}=\frac{\eta_3\eta_{13}}{\eta_1\eta_{39}}$ is a Hauptmodul for $\Gamma$. Choose
	\[x(\tau)=1/t_{\Gamma}=q-q^2-q^3+q^4-q^5+2q^7+\cdots,\]
	According to \cite{TDDN}, we have
	\begin{equation}\label{wfor39+39}
		w(x)=(1+x)^2(1-7x+11x^2-7x^3+x^4)(1+x-x^2+x^3+x^4),
	\end{equation}
	\[R(x)=x(2+17x-48x^2-25x^3+194x^4-45x^5-168x^6+137x^7+82x^8-25x^9),\]
	and
	\[z(\tau)=\frac{(\log(x))_q}{\sqrt{w(x)}}=1+q+3q^2+q^3+5q^4+3q^5+7q^6+5q^7+\cdots.\]
	Now assume $\displaystyle z=\sum\limits_{n=0}^\infty A_nx^n$, via Fourier coefficients of $x$ and $z$, we have
	\begin{equation*}
		\begin{split}
			&A_{0}=1, A_{1}=1, A_{2}=4, A_{3}=10, A_{4}=38, A_{5}=140, \\
			&A_{6}=563, A_{7}=2315, A_{8}=9816, A_{9}=42432.
		\end{split}
	\end{equation*}
	Using (\ref{DE}), we can derive the recurrence relation
	\begin{equation*}
		\begin{split}
			&\left(-250+150 n-30 n^{2}+2 n^{3}\right) A_{n-10}+\left(738-488 n+108 n^{2}-8n^{3}\right) A_{n-9} \\
			+&\left(1096-786 n+192 n^{2}-16 n^{3}\right) A_{n-8}+\left(-1176+924 n-252 n^{2}+24 n^{3}\right) A_{n-7}\\
			+&\left(-270+234 n-72 n^{2}+8 n^{3}\right) A_{n-6}+\left(970-938 n+330 n^{2}-44 n^{3}\right) A_{n-5}\\
			+&\left(-100+114 n-48 n^{2}+8 n^{3}\right) A_{n-4}+\left(-144+204 n-108 n^{2}+24 n^{3}\right) A_{n-3} \\
			+&\left(34-66 n+48 n^{2}-16 n^{3}\right) A_{n-2}+\left(2-8 n+12 n^{2}-8 n^{3}\right) A_{n-1}+2 n^{3} A_{n}=0.
		\end{split}
	\end{equation*}
\end{example}

\begin{remark}
	10 moonshine groups together with their Hauptmodul $t_\Gamma$, $x(\tau),w(x),R(x)$ that we used to derive $1/\pi$-series are listed in Table \ref{HXWR} in Section \ref{Tables}. We also list the recurrence relations of $A_n$ implied by (\ref{DE}) with initial values in Table \ref{RR}.
\end{remark}

\section{Modular equations for $x(\tau)$}

Recall that the $j$-invariant
\[j(\tau)=\frac{1}{q}+744+196884q+\cdots\]
has the property that for any CM-point $\tau_0\in\mathcal{H}$ (i.e. There exist integers $a, b, c$ with $a \neq 0$ such that $a\tau_0^2+b\tau_0+c=0$), $j(\tau_0)$ is algebraic. This fact can be proved by using the modular equations for $j(\tau)$, that are polynomials $\Psi_n(X,Y)\in\Z[X,Y]$ such that
\[\Psi_n(j(\tau),Y)=\prod_{i=1}^r(Y-j(\gamma_i\tau)),\]
where $\displaystyle \gamma_i\in\Gamma_{n}:=\left\{\begin{pmatrix}
	a & b \\
	c & d
\end{pmatrix}\in M_{2\times 2}(\Z)|ad-bc=n\right\}$ and $\displaystyle \Gamma_n=\bigsqcup_{i=1}^{r}\mathrm{SL}_2(\Z)\gamma_i$. Similarly, for each Moonshine group $\Gamma$ and its Hauptmodul $x(\tau)$ we have

\begin{proposition}[Proposition 3.2 of \cite{TDDN}]\label{MEQ}
	Assume $x(\tau)$ has the Fourier expansion of the form $\displaystyle x(\tau)=\frac{1}{q}+O(1)$. For any integer $n \geqslant 2$ with $\gcd(n,N)=1,$ there is a symmetric irreducible polynomial $\Psi_n(X,Y)\in\CC[X,Y]$ of degree $\psi(n)=n\prod_{\substack{q\mid n\\q\  \text{prime}}}(1+\frac{1}{q})$ in $X$ and $Y$ such that
	\begin{equation}\label{Modular equation}
		\Psi\left(x(\tau),Y\right)=\prod\limits_{
			\substack{
				\alpha \delta=n \\
				0 \leq \beta< \delta \\
				\text{gcd}(\alpha,\beta,\delta)=1
		}}
		\left(Y-x\left(\frac{\alpha\tau+\beta}{\delta}\right)\right).
	\end{equation}
That is to say, if $S_k$ is the elementary symmetric function of degree $1\leqslant k\leqslant \psi(n)$ in $\psi(n)$ variables, then $\displaystyle S_k\left(x(n\tau),\cdots,x\left(\frac{\tau+n-1}{n}\right)\right)$ is a polynomial in $x(\tau)$.
\end{proposition}

To use Proposition \ref{MEQ}, we require that $x(\tau)$ has the Fourier expansion of the form $\displaystyle \frac{1}{q}+O(1)$. However, the Möbius transformations $x(\tau)=1/t_{\Gamma}$ that we chosen in Section \ref{HaDE} all have the Fourier expansion $x(\tau)=q+O(q^2)$. Fortunately, it turns out that in this cases, we can also find the Modular Equations $\Psi(X,Y)$. For example, we have
\begin{theorem}\label{MFFOR39+39}
	If $\displaystyle\Gamma=39+39,t_{\Gamma}=\frac{\eta_3\eta_{13}}{\eta_1\eta_{39}}=\frac{1}{q}+O(1), x(\tau)=1/t_\Gamma$. Let
	\[\Psi_2(X,Y)=Y^3+(2X-X^2)Y^2+(-X+2X^2)Y+X^3,\]
	then we have
	\begin{equation}\label{MEQ39+39}
		\Psi_2(x(\tau),Y)=\left(Y-x\left(2\tau\right)\right)\left(Y-x\left(\frac{\tau}{2}\right)\right)\left(Y-x\left(\frac{\tau+1}{2}\right)\right).
	\end{equation}
\end{theorem}
\begin{proof}
	Let $\displaystyle t_1=t_\Gamma(2\tau),t_2=t_\Gamma\left(\frac{\tau}{2}\right),t_3=t_\Gamma\left(\frac{\tau+1}{2}\right)$. Proposition \ref{MEQ} says that, $t_1+t_2+t_3,t_1t_2+t_1t_3+t_2t_3$ and $t_1t_2t_3$ are all polynomial of degree small than $\psi(2)=3$ in $t_\Gamma$. By using the $q$-expansion, we find that
	\[t_1+t_2+t_3=t_\Gamma^2-2t_\Gamma,\ \ t_1t_2+t_1t_3+t_2t_3=2t_\Gamma^2-t_\Gamma,\ \ t_1t_2t_3=-t_\Gamma^3.\]
	Since $x(\tau)=1/t_\Gamma(\tau),$ we have
	\begin{align*}
		x\left(2\tau\right)+x\left(\frac{\tau}{2}\right)+x\left(\frac{\tau+1}{2}\right)&=\frac{t_1t_2+t_1t_3+t_2t_3}{t_1t_2t_3}\\
		&=\frac{2t_\Gamma^2-t_\Gamma}{-t_\Gamma^3}\\ 
		&=-2x+x^2.
	\end{align*}
	Similarly, 
	\[x\left(2\tau\right)x\left(\frac{\tau}{2}\right)+x\left(2\tau\right)x\left(\frac{\tau+1}{2}\right)+x\left(\frac{\tau}{2}\right)x\left(\frac{\tau+1}{2}\right)=-x+2x^2,\]
	\[x\left(2\tau\right)x\left(\frac{\tau}{2}\right)x\left(\frac{\tau+1}{2}\right)=-x^3.\]
	Thus, we have (\ref{MEQ39+39}).
\end{proof}

For each $\Gamma$ and $x(\tau)=1/t_\Gamma$, we choose $n=2$ or $3$ and list their modular equations $\Psi_n(X,Y)$ in Table \ref{MEQOH}.

\section{Series for $1/\pi$}


Suppose that $\Phi_n(X,Y)$ is a modular equation for some $\Gamma$. If $\tau_0\in \mathbb{H}$ such that
\[\frac{a\tau_0+b}{c\tau_0+d}=\frac{\alpha\tau_0+\beta}{\delta}\]
for some $a\delta=n, 0\leqslant \beta<\delta$ and $\displaystyle\begin{pmatrix}
	a & b \\
	c & d
\end{pmatrix}\in\Gamma$, then by (\ref{Modular equation}), we know that
\begin{align*}
	\Psi_n\left(x\left(\tau_0\right),x(\tau_0)\right)&=\Psi_n\left(x(\tau_0),x\left(\frac{a\tau_0+b}{c\tau_0+d}\right)\right)\\
	&=\Psi_n\left(x(\tau_0),x\left(\frac{\alpha\tau_0+\beta}{\delta}\right)\right)\\ 
	&=0.
\end{align*}
Thus $x(\tau_0)$ is a root of $\Psi_n(X,X)=0$.

We now derive Ramanujan-Sato series for $1/\pi$, the following argument follows \cite{TDDN}. If
\[M=\begin{pmatrix}
	a' & b' \\
	c' & d'
\end{pmatrix}:=\begin{pmatrix}
	a & b \\
	c & d
\end{pmatrix}^{-1}\begin{pmatrix}
\alpha & \beta \\
0 & \delta
\end{pmatrix}\]
has $c'\neq0$ (this is equivalent to $c\neq0$), then
$M\tau_0=\tau_0$ and $y(\tau):=x\left( M\tau \right)$ satisfies
\begin{equation}\label{PSIXY}
	\Psi_n(x(\tau),y(\tau))=0
\end{equation}
by (\ref{Modular equation}). We can apply (\ref{PSIXY}) to expand $y$ about $x(\tau_0)$:
\begin{equation}\label{yexpandaboutx}
	y(\tau)=\sum_{k=0}^\infty y^{(k)}(x(\tau_0))\frac{\left(x(\tau)-x(\tau_0)\right)^k}{k!},
\end{equation}
where $\displaystyle y^{(k)}=\frac{\mathrm{d}^ky}{\mathrm{d}x^k}$. Now applying $\displaystyle \frac{1}{2\pi\mathrm{i}}\frac{\mathrm{d}}{\mathrm{d}\tau}$ to both side of (\ref{yexpandaboutx}), we have
\begin{equation*}
	\frac{1}{2\pi\mathrm{i}}\frac{a'd'-b'c'}{(c'\tau+d')^2}\cdot\frac{\mathrm{d}x}{\mathrm{d}\tau}(M\tau)=\frac{1}{2\pi\mathrm{i}}\left(y^{(1)}(x(\tau_0))\frac{\mathrm{d}x}{\mathrm{d}\tau}(\tau)+y^{(2)}(x(\tau_0))(x(\tau)-x(\tau_0))\frac{\mathrm{d}x}{\mathrm{d}\tau}(\tau)+\cdots\right).
\end{equation*}
Let $\displaystyle z=\frac{(\log(x))_q}{\sqrt{w(x)}}$ as in Theorem \ref{MDFz}. Then 
\begin{equation}\label{xzwtau}
	\frac{1}{2\pi\mathrm{i}}\frac{\mathrm{d}x}{\mathrm{d}\tau}(\tau)=x(\tau)z(\tau)\sqrt{w(\tau)}.
\end{equation}
Denote $\sqrt{w(\tau)}$ by $W(\tau)$, thus
\begin{equation}\label{diff1}
\frac{a'd'-b'c'}{(c'\tau+d')^2}	x(M\tau)z(M\tau)W(M\tau)=\left(y^{(1)}(x(\tau_0))+y^{(2)}(x(\tau_0))(x-x(\tau_0))+\cdots\right)xzW.
\end{equation}
Applying $\displaystyle \frac{1}{2\pi\mathrm{i}}\frac{\mathrm{d}}{\mathrm{d}\tau}$ again to both side of (\ref{diff1}) and use (\ref{xzwtau}), we have
\begin{equation}
	\label{diff2}
	\begin{split}
		\frac{\mathrm{i}c'}{\pi}\frac{a'd'-b'c'}{(c'\tau+d')^2}&x(M\tau)z(M\tau)W(M\tau)+\frac{(a'd'-b'c')^2}{(c'\tau+d')^4}\Big(z(M\tau)W(M\tau)+\\
		x(M\tau)\frac{\mathrm{d}z}{\mathrm{d}x}&(M\tau)W(M\tau)+x(M\tau)z(M\tau)\frac{\mathrm{d}W}{\mathrm{d}x}(M\tau)\Big)x(\tau)z(M\tau)W(M\tau)\\
		=&\left(y^{(2)}(x(\tau_0))+\cdots\right)x^2z^2W^2+\big(y^{(1)}(x(\tau_0))+y^{(2)}(x(\tau_0))(x-x(\tau_0))\\
		&+\cdots\big)\left(z(\tau)w(\tau)+x(\tau)\frac{\mathrm{d}z}{\mathrm{d}x}(\tau)W(\tau)+x(\tau)z(\tau)\frac{\mathrm{d}W}{\mathrm{d}x}(\tau)\right)xzW.
	\end{split}
\end{equation}
Let $\tau=\tau_0$ in (\ref{diff2}) and make some simplification, we have
\[\frac{1}{\pi}=\left.\frac{W(1-y^{(1)})(c'\tau_0+d')}{\mathrm{i}c'}\left(x\frac{\mathrm{d}z}{\mathrm{d}x}+\left(1+\frac{x\frac{\mathrm{d}W}{\mathrm{d}x}}{W}+\frac{xy^{(2)}}{y^{(1)}(1-y^{(1)})}\right)z\right)\right|_{\tau=\tau_0}.\]
Since $\displaystyle z=\sum\limits_{n=0}^\infty A_nx^n$, we get
\[\frac{1}{\pi}=\sum_{n=0}^{\infty}A_n(Bn+C)x(\tau_0)^n,\]
where
\begin{equation}\label{BandC}
	\begin{split}
		&B=\frac{W(x(\tau_0))(1-y^{(1)}(x(\tau_0))(c'\tau_0+d'))}{\mathrm{i}c'},\\
		&C=B\left(1+\frac{x(\tau_0)\frac{\mathrm{d}W}{\mathrm{d}x}(x(\tau_0))}{W(x(\tau_0))}+\frac{x(\tau_0)y^{(2)}(x(\tau_0))}{y^{(1)}(x(\tau_0))(1-y^{(1)}(x(\tau_0)))}\right).
	\end{split}
\end{equation}

Let's continue with the example $\Gamma=39+39$ to illustrate how the above calculation works.

\begin{example}
	Let $n=2$, we know from Theorem \ref{MFFOR39+39} that the modular equation for $\displaystyle x=\frac{\eta_1\eta_{39}}{\eta_3\eta_{13}}$ is
	\begin{equation}\label{meqfor39+39}
		\Psi_2(X,Y)=Y^3+(2X-X^2)Y^2+(-X+2X^2)Y+X^3.
	\end{equation}
	Thus, the roots of $\Psi(X,X)=0$ are $0,3\pm2\sqrt{2}$. Let $\displaystyle \tau_0=\mathrm{i}\sqrt{\frac{2}{39}}$, notice that $\tau_0$ satisfies
	\[\begin{pmatrix}
		0 & -1 \\
		39 & 0
	\end{pmatrix}\tau_0=\begin{pmatrix}
	1 & 0 \\
	0 & 2
\end{pmatrix}\tau_0,\]
where $\begin{pmatrix}
	0 & -1 \\
	39 & 0
\end{pmatrix}\in\Gamma$. Choose $M=\begin{pmatrix}
0 & -1 \\
39 & 0
\end{pmatrix}^{-1}\begin{pmatrix}
1 & 0 \\
0 & 2
\end{pmatrix}=\begin{pmatrix}
0 & 2/39 \\
-1 & 0
\end{pmatrix}$, we have $c'=-1,d'=0$, set $y(\tau)=x(M\tau)$. By numerical approximation, we have $x(\tau_0)=3-2\sqrt{2}$. Since $\Psi_2(x(\tau),y(\tau))=0$, we can calculate from (\ref{meqfor39+39}) that
\[y^{(1)}(3-2\sqrt{2})=-1,\ \ y^{(2)}(3-2\sqrt{2})=-16-\frac{25}{\sqrt{2}}.\]
We can also calculate from (\ref{wfor39+39}) that
\[W(3-2\sqrt{2})=-6(-2+\sqrt{2})\sqrt{7501-5304\sqrt{2}},\]
\[\frac{\mathrm{d}W}{\mathrm{d}x}(3-2\sqrt{2})=-3(-5753+4068\sqrt{2})\sqrt{7501+5304\sqrt{2}}.\]
Finally, via (\ref{BandC}), we get $\displaystyle \frac{1}{\pi}=\sum_{n=0}^{\infty}A_n(Bn+C)\big(3-2\sqrt{2}\big)^n$ with
\[B=-4(-2+\sqrt{2})\sqrt{6(577-408 \sqrt{2})},\]
\[C=44 \sqrt{3(577-408 \sqrt{2})}-22 \sqrt{6(577-408 \sqrt{2})}+2(-47420+33531\sqrt{2})\sqrt{3(577+408\sqrt{2})},\]
and $A_n$ as in Example \ref{Anfor39+39}.
\end{example}

In Table \ref{tau0Mx} we list $\tau_0, M, x(\tau_0)$ that used to derive the series. Also in Table \ref{BCFORSERIES}, we list $B,C$ for each series.

\section{Tables}\label{Tables}
In this section, we list the tables of modular equations of Huaptmodul $t_{\Gamma}$, the $a,b,c$ for $1/\pi$ series $\frac{1} {\pi}=\sum _{n=0}^{\infty }A_{ n}\left(an+b \right)c^{n}$, and recursion formula of $A_n$ for some given level $N$.

The details of $R(x)$ and $w(x)$ defining the differential equation come from \cite{TDDN}. The notation $\eta_{\alpha}$ means
$$\eta_a=q^{a/24}\prod\limits_{n=1}^{\infty}(1-q^{an}).$$
Moreover, we write $N+e$ for $\Gamma_{0}(N)+e$.

\begin{table}[htbp]
	\caption{Huaptmoduls, $x(\tau), w(x)$ and $R(x)$ for 10 moonshine groups. \label{HXWR}}
	\begin{tabular}{|c|c|c|c|}
		\hline
		\multirow{2}*{\vspace{-0.2cm}$\Gamma$} & \multirow{2}*{\vspace{-0.2cm}$t_{\Gamma}$} & \multirow{2}*{\vspace{-0.2cm}$x$}  & \gape{$w(x)$}\\
		\cline{4-4}
		& & & \gape{$R(x)$} \\
		\hline
		\multirow{2}*{$14+7$} & \multirow{2}*{$\big(\frac{\eta_1\eta_7}{\eta_2\eta_{14}}\big)^3 $} & \multirow{20}*{\vspace{-2.5cm}$1/t_{\Gamma}$}  & \gape{$(1+x)(1+8x)(1+5x+8x^2)$}\\
		\cline{4-4}
		& & & \gape{$-8x(1+4x)(1+7x+8x^2)$} \\
		\cline{1-2}\cline{4-4}
		\multirow{2}*{$14+14$} & \multirow{2}*{$\big(\frac{\eta_2\eta_7}{\eta_1\eta_{14}}\big)^4$} &  & \gape{$1-14x+19x^2-14x^3+x^4$}\\
		\cline{4-4}
		& & & \gape{$x(6-25x+34x^2-4x^3)$} \\
		\cline{1-2}\cline{4-4}
		\multirow{2}*{$15+15$} & \multirow{2}*{$\big(\frac{\eta_1\eta_5}{\eta_3\eta_{15}}\big)^3$} &  & \gape{$(-1-x+x^2)(-1+11x+x^2)$}\\
		\cline{4-4}
		& & & \gape{$4x(1+4x-6x^2-x^3)$} \\
		\cline{1-2}\cline{4-4}
		\multirow{2}*{$16+$} & \multirow{2}*{$\frac{(\eta_2\eta_8)^6}{(\eta_1\eta_4\eta_{16})^4} $} &  & \gape{$(1-2x)^2(1-12x+4x^2)$}\\
		\cline{4-4}
		& & & \gape{$8x(1-2x)(1-8x+4x^2)$} \\
		\cline{1-2}\cline{4-4}
		\multirow{2}*{$20+20$} & \multirow{2}*{$\big(\frac{\eta_4\eta_5}{\eta_1\eta_{20}}\big)^2$} &  & \gape{$(1+x)^2(1-8x-2x^2-8x^3+x^4)$}\\
		\cline{4-4}
		& & & \gape{$x(1+x)(2+25x+31x^2+47x^3-9x^4)$} \\
		\cline{1-2}\cline{4-4}
		\multirow{2}*{$21+21$} & \multirow{2}*{$\big(\frac{\eta_3\eta_7}{\eta_1\eta_{21}}\big)^2$} &   & \gape{$(1-x)^2(1-6x-17x^2-6x^3+x^4)$}\\
		\cline{4-4}
		& & & \gape{$4x+4x^2-70x^3+16x^4+52x^5-9x^6$} \\
		\cline{1-2}\cline{4-4}
		\multirow{2}*{$22+11$} & \multirow{2}*{$\big(\frac{\eta_1\eta_{11}}{\eta_2\eta_{22}}\big)^2$} &  & \gape{$(1+4x+8x^2+4x^3)(1+8x+16x^2+16x^3)$}\\
		\cline{4-4}
		& & & \gape{$-8x(1+12x+57x^2+132x^3+160x^4+72x^5)$} \\
		\cline{1-2}\cline{4-4}
		\multirow{2}*{$26+26$} & \multirow{2}*{$\big(\frac{\eta_2\eta_{13}}{\eta_1\eta_{26}}\big)^2$} &   & \gape{$(1-x)(1-8x+8x^2-18x^3+8x^4-8x^5+x^6)$}\\
		\cline{4-4}
		& & & \gape{\hspace{2.5mm}$(x/4)(20-109x+339x^2-521x^3+445x^4-335x^5+49x^6)$\hspace{0.5mm}\  } \\
		\cline{1-2}\cline{4-4}
		\multirow{2}*{$35+35$} & \multirow{2}*{$\frac{\eta_5\eta_7}{\eta_1\eta_{35}}$} &  & \gape{$(1+x-x^2)(1-5x-9x^3-5x^5-x^6)$}\\
		\cline{4-4}
		& & & \gape{$-x(-2-9x-14x^2-47x^3+30x^4-57x^5+50x^6+16x^7)$} \\
		\cline{1-2}\cline{4-4}
		\multirow{2}*{\vspace{-0.7cm}$39+39$} & \multirow{2}*{\vspace{-0.7cm}$\frac{\eta_3\eta_{13}}{\eta_1\eta_{39}}$} &  & \gape{$(1+x)^2(1-7x+11x^2-7x^3+x^4)(1+x-x^2+x^3+x^4)$}\\
		\cline{4-4}
		& & &
		\gape{\tabincell{c}{$x(2+17x-48x^2-25x^3+194x^4-45x^5-168x^6+137x^7$\\$+82x^8-25x^9)$}}  \\
		\hline
	\end{tabular}
\end{table}

\begin{table}[htbp]
	\caption{The recurrence relations and initial values of $A_n$.\label{RR}}
	\centering
	\begin{tabular}{|c|c|}
		\hline
		\gape{$\Gamma$} & \gape{$A_n$} \\
		\hline
		14+7&\gape{\tabincell{c}{
				$\left(-1024+1536 n-768 n^{2}+128 n^{3}\right) A_{n-4}+(-864+1584 n-1008 n^{2}$ \\
				$+224 n^{3}) A_{n-3}+\left(-176+420 n-366 n^{2}+122 n^{3}\right) A_{n-2}+(-8+30n $\\
				$-42 n^{2}+28 n^{3}) A_{n-1}+2n^{3} A_{n}=0,$\\
				 $A_{0}=1, A_{1}=-4, A_{2}=16, A_{3}=-72.$}} \\
		\hline
		14+14& \gape{\tabincell{c}{
				$\left(-16+24 n-12 n^{2}+2 n^{3}\right) A_{n-4}+\left(102-194 n+126 n^{2}-28 n^{3}\right) A_{n-3}+(-50$\hspace{1mm}\\
				$+126 n-114 n^{2}+38 n^{3}) A_{n-2}+\left(6-26 n+42 n^{2}-28 n^{3}\right) A_{n-1}+2n^{3}A_{n}=0,$\\
				$A_{0}=1, A_{1}=3, A_{2}=16, A_{3}=117.$}}\\
		\hline
		15+15& \gape{\tabincell{c}{$\left(-16+24 n-12 n^{2}+2 n^{3}\right) A_{n-4}+\left(-72+138 n-90 n^{2}+20 n^{3}\right) A_{n-3}$ \\
				$+\left(32-84 n+78 n^{2}-26 n^{3}\right) A_{n-2}+\left(4-18 n+30 n^{2}-20 n^{3}\right) A_{n-1}$\\
				$+2 n^{3} A_{n}=0,$\\
				$A_{0}=1, A_{1}=2, A_{2}=11, A_{3}=72.$}}\\
		\hline
		16+& \gape{\tabincell{c}{
				$\left(-256+384 n-192 n^{2}+32 n^{3}\right) A_{n-4}+\left(480-896 n+576 n^{2}-128 n^{3}\right) A_{n-3}$ \\
				$+\left(-160+384 n-336 n^{2}+112 n^{3}\right) A_{n-2}+\left(8-32 n+48 n^{2}-32 n^{3}\right) A_{n-1}$\\
				$+2 n^{3} A_{n}=0,$\\
				$A_{0}=1, A_{1}=4, A_{2}=20, A_{3}=128.$}}\\
		\hline
		20+20& \gape{\tabincell{c}{
				$\left(-54+54 n-18 n^{2}+2 n^{3}\right) A_{n-6}+\left(190-226 n+90 n^{2}-12 n^{3}\right) A_{n-5}+(312$\\
				$-428 n+204 n^{2}-34 n^{3}) A_{n-4}+\left(168-292 n+180 n^{2}-40 n^{3}\right) A_{n-3}+(54$ \\
				$-122 n+102 n^{2}-34 n^{3}) A_{n-2}+\left(2-10 n+18 n^{2}-12 n^{3}\right) A_{n-1}+2 n^{3} A_{n}=0,$\\
				$A_{0}=1, A_{1}=1, A_{2}=6, A_{3}=30, A_{4}=175, A_{5}=1087.$}}\\
		\hline
		21+21& \gape{\tabincell{c}{
				$\left(-54+54 n-18 n^{2}+2 n^{3}\right) A_{n-6}+\left(260-304 n+120 n^{2}-16 n^{3}\right) A_{n-5} $\\
				$+\left(64-96 n+48 n^{2}-8 n^{3}\right) A_{n-4}+\left(-210+338 n-198 n^{2}+44 n^{3}\right) A_{n-3}$ \\
				\hspace{1.5mm}$+\left(8-24 n+24 n^{2}-8 n^{3}\right) A_{n-2}+\left(4-16 n+24 n^{2}-16 n^{3}\right) A_{n-1}+2 n^{3} A_{n}=0,$\hspace{1mm}\  \\
				$A_{0}=1, A_{1}=2, A_{2}=8, A_{3}=37, A_{4}=204, A_{5}=1218.$}}\\
		\hline
		22+11 & \gape{\tabincell{c}{
				$\left(-3456+3456n-1152 n^{2}+128 n^{3}\right) A_{n-6}+(-6400+7360 n-2880 n^{2}$\\
				$+384 n^{3}) A_{n-5}+\left(-4224+5696 n-2688 n^{2}+448 n^{3}\right) A_{n-4}+(-1368$ \\
				$+2244 n-1332 n^{2}+296 n^{3}) A_{n-3}+\left(-192+416 n-336 n^{2}+112 n^{3}\right) A_{n-2}$ \\
				$+(-8+28 n-36 n^{2}+24 n^{3}) A_{n-1}+2 n^{3} A_{n}=0,$\\
				$A_{0}=1, A_{1}=-4, A_{2}=12, A_{3}=-36, A_{4}=124, A_5=-496.$}}\\
		\hline
		26+26&\gape{\tabincell{c}{
				$\left(\frac{343}{4}-\frac{147 n}{2}+21 n^{2}-2 n^{3}\right) A_{n-7}+\left(-\frac{1005}{2}+\frac{983 n}{2}-162 n^{2}+18 n^{3}\right) A_{n-6}$\\
				$+\left(\frac{2225}{4}-\frac{1245 n}{2}+240 n^{2}-32 n^{3}\right) A_{n-5}
				+\left(-521+\frac{1353 n}{2}-312 n^{2}+52 n^{3}\right) A_{n-4}$ \\
				$+\left(\frac{1017}{4}-\frac{807 n}{2}+234 n^{2}-52 n^{3}\right) A_{n-3}+\left(-\frac{109}{2}+\frac{237 n}{2}-96 n^{2}+32 n^{3}\right) A_{n-2}$ \\
				$+\left(5-19 n+27 n^{2}-18 n^{3}\right) A_{n-1}+2 n^{3} A_{n}=0,$\\
				$A_{0}=1, A_{1}=\frac{5}{2}, A_{2}=\frac{59}{8}, A_{3}=\frac{497}{16},A_{4}=\frac{19539}{128}, A_{5}=\frac{207051}{256}, A_{6}=\frac{4623151}{1024}.$}}\\
		\hline
		35+35&\gape{\tabincell{c}{
				$\left(-128+96 n-24 n^{2}+2 n^{3}\right) A_{n-8}+\left(-350+296 n-84 n^{2}+8 n^{3}\right) A_{n-7}$ \\
				$+\left(342-330 n+108 n^{2}-12 n^{3}\right) A_{n-6}+\left(-150+160 n-60 n^{2}+8 n^{3}\right) A_{n-5}$ \\
				$+\left(188-238 n+108 n^{2}-18 n^{3}\right) A_{n-4}+\left(42-64 n+36 n^{2}-8 n^{3}\right) A_{n-3}$ \\
				$+\left(18-42 n+36 n^{2}-12 n^{3}\right) A_{n-2}+\left(2-8 n+12 n^{2}-8 n^{3}\right) A_{n-1}+2 n^{3} A_{n}=0,$\\
				$A_{0}=1, A_{1}=1, A_{2}=3, A_{3}=10, A_{4}=38, A_{5}=150, A_{6}=627, A_{7}=2703. $}}\\
		\hline
	\end{tabular}
\end{table}

\begin{table}[htbp]
	\centering
	\begin{tabular}{|c|c|}
		\hline
		39+39& \gape{\tabincell{c}{$\left(-250+150 n-30 n^{2}+2 n^{3}\right) A_{n-10}+\left(738-488 n+108 n^{2}-8n^{3}\right) A_{n-9} $\\
				\hspace{1mm}$+\left(1096-786 n+192 n^{2}-16 n^{3}\right) A_{n-8}+\left(-1176+924 n-252 n^{2}+24 n^{3}\right) A_{n-7}$\hspace{0.8mm} \\
				$+\left(-270+234 n-72 n^{2}+8 n^{3}\right) A_{n-6}+\left(970-938 n+330 n^{2}-44 n^{3}\right) A_{n-5}$ \\
				$+\left(-100+114 n-48 n^{2}+8 n^{3}\right) A_{n-4}+\left(-144+204 n-108 n^{2}+24 n^{3}\right) A_{n-3}$ \\
				$+\left(34-66 n+48 n^{2}-16 n^{3}\right) A_{n-2}+\left(2-8 n+12 n^{2}-8 n^{3}\right) A_{n-1}+2 n^{3} A_{n}=0,$\\
				$A_{0}=1, A_{1}=1, A_{2}=4, A_{3}=10, A_{4}=38, A_{5}=140, A_{6}=563, A_{7}=2315, $\\
				$A_{8}=9816, A_{9}=42432.$}}\\
		\hline
	\end{tabular}
\end{table}

\begin{table}[htbp]
	\caption{The modular equation of Huaptmodul.\label{MEQOH}}
	\begin{tabular}{|c|c|}
		\hline
		\gape{$\Gamma$}  & \gape{$\psi_n(X,Y),n$}\\
		\hline
		14+7&\gape{\tabincell{c}{\hspace{1mm}$Y^{4}+\left(-18 X-72 X^{2}-64 X^{3}\right)Y^{3}+\left(-9 X-54 X^{2}-72 X^{3}\right)Y^{2}+(-X-9 X^{2}$\\$-18 X^{3}) Y+X^{4}$, $n=3$}}\\
		\hline
		14+14&\gape{\tabincell{c}{$Y^{4}+\left(-18 X+12 X^{2}-X^{3}\right)Y^{3}+\left(12 X+9 X^{2}+12 X^{3}\right)Y^{2}+(-X+12 X^{2}$\\$-18 X^{3}) Y+X^{4},n=3$}}\\
		\hline
		15+15& \gape{$Y^{3}+\left(6 X+X^{2}\right)Y^{2}+\left(-X+6 X^{2}\right) Y+X^{3},n=2$}\\
		\hline
		16+& \gape{\tabincell{c}{$Y^{4}+\left(-24 X+48 X^{2}-16 X^{3}\right)Y^{3}+\left(12 X-42 X^{2}+48 X^{3}\right)Y^{2}+(-X+12 X^{2}$\\$-24 X^{3}) Y+X^{4},n=3$}}\\
		\hline
		20+20& \gape{\tabincell{c}{$Y^{4}+\left(3 X+6 X^{2}-X^{3}\right)Y^{3}+\left(6 X+18 X^{2}+6 X^{3}\right)Y^{2}+(-X+6 X^{2}+3 X^{3}) Y$\\$+X^{4},n=3$}}\\
		\hline
		21+21& \gape{$Y^{3}+\left(4 X-X^{2}\right)Y^{2}+\left(-X+4 X^{2}\right) Y+X^{3},n=2$}\\
		\hline
		22+11& \gape{\tabincell{c}{$Y^{4}+\left(-9 X-24 X^{2}-16 X^{3}\right)Y^{3}+\left(-6 X-24 X^{2}-24 X^{3}\right)Y^{2}+(-X-6 X^{2}$\\$-9 X^{3}) Y+X^{4},n=3$}}\\
		\hline
		26+26& \gape{\tabincell{c}{$Y^{4}+\left(-3 X+6 X^{2}-X^{3}\right)Y^{3}+\left(6 X-9 X^{2}+6 X^{3}\right)Y^{2}+\left(-X+6 X^{2}-3 X^{3}\right)Y\hspace{-0mm}$\\$ +X^{4},n=3$}}\\
		\hline
		35+35& \gape{$Y^{3}+\left(2 X+X^{2}\right)Y^{2}+\left(-X+2 X^{2}\right) Y+X^{3},n=2$}\\
		\hline
		39+39& \gape{$Y^{3}+\left(2 X-X^{2}\right)Y^{2}+\left(-X+2 X^{2}\right) Y+X^{3},n=2$}\\
		\hline
	\end{tabular}
\end{table}

\begin{table}[htbp]
	\caption{$\tau_0, M$ and $x(\tau_0)$ that used to derive the series for $1/\pi$.\label{tau0Mx}}
	\begin{tabular}{|c|c|c|c|c|}
		\hline
		\gape{$\Gamma$} & \gape{$\tau_0$} & \gape{$\begin{pmatrix}
				a & b \\
				c & d
			\end{pmatrix}\tau_0=\begin{pmatrix}
				\alpha & \beta \\
				0 & \delta
			\end{pmatrix}\tau_0$} &\gape{$M$} & \gape{$x(\tau_0)$} \\
		\hline
		14+7& \gape{$\displaystyle \frac{-7+\mathrm{i}\sqrt{21}}{14}$} &\gape{$\begin{pmatrix}
				-7 & -3 \\
				-14 & -7
			\end{pmatrix}\tau_0=\begin{pmatrix}
				1 & 2 \\
				0 & 3
			\end{pmatrix}\tau_0$} &\gape{$\begin{pmatrix}
			1 & -5/7 \\
			2 & 1
		\end{pmatrix}$} &\gape{$\displaystyle \frac{1}{4}(-3+\sqrt{7})$} \\
		\hline
		14+14& \gape{$\displaystyle \mathrm{i}\sqrt{\frac{3}{14}}$} &\gape{$\begin{pmatrix}
				0 & 1 \\
				-14 & 0
			\end{pmatrix}\tau_0=\begin{pmatrix}
				1 & 0 \\
				0 & 3
			\end{pmatrix}\tau_0$} &\gape{$\begin{pmatrix}
			0 & -3/14 \\
			1 & 0
		\end{pmatrix}$} &\gape{$\displaystyle \frac{2}{23+5\sqrt{21}}$}\\
		\hline
		15+15& \gape{$\displaystyle \mathrm{i}\sqrt{\frac{2}{15}}$} &\gape{$\begin{pmatrix}
				0 & -1 \\
				15 & 0
			\end{pmatrix}\tau_0=\begin{pmatrix}
				1 & 0 \\
				0 & 2
			\end{pmatrix}\tau_0$} &\gape{$\begin{pmatrix}
			0 & 2/15 \\
			-1 & 0
		\end{pmatrix}$}  & \gape{$\displaystyle -7+5\sqrt{2}$}\\
		\hline
		16+& \gape{$\displaystyle \frac{8+\mathrm{i}\sqrt{3}}{4}$} &\gape{$\begin{pmatrix}
				-16 & 33 \\
				-16 & 32
			\end{pmatrix}\tau_0=\begin{pmatrix}
				1 & 1 \\
				0 & 3
			\end{pmatrix}\tau_0$} &\gape{$\begin{pmatrix}
			2 & -67/16 \\
			1 & -2
		\end{pmatrix}$}  &\gape{$\displaystyle \frac{1}{2}(5-2\sqrt{6})$}\\
		\hline
		20+20& \gape{$\displaystyle \frac{1}{2}\mathrm{i}\sqrt{\frac{3}{5}}$} &\gape{$\begin{pmatrix}
				0 & -1 \\
				20 & 0
			\end{pmatrix}\tau_0=\begin{pmatrix}
				1 & 0 \\
				0 & 3
			\end{pmatrix}\tau_0$} &\gape{$\begin{pmatrix}
			0 & 3/20 \\
			-1 & 0
		\end{pmatrix}$}  &\gape{$\displaystyle 7-4\sqrt{3}$}\\
		\hline
		21+21& \gape{$\displaystyle \mathrm{i}\sqrt{\frac{2}{21}}$} &\gape{$\begin{pmatrix}
				0 & 1 \\
				-21 & 0
			\end{pmatrix}\tau_0=\begin{pmatrix}
				1 & 0 \\
				0 & 2
			\end{pmatrix}\tau_0$} &\gape{$\begin{pmatrix}
			0 & -2/21 \\
			1 & 0
		\end{pmatrix}$}  &\gape{$\displaystyle 5-2\sqrt{6}$}\\
		\hline
		22+11 & \gape{\hspace{0.5mm}$\displaystyle \frac{-33+\mathrm{i}\sqrt{33}}{22}$\hspace{0.5mm}\ }&\gape{\hspace{0.5mm}$\begin{pmatrix}
				-11 & -17 \\
				22 & 33
			\end{pmatrix}\tau_0=\begin{pmatrix}
				1 & 0 \\
				0 & 3
			\end{pmatrix}\tau_0\hspace{0.5mm}$} &\gape{$\begin{pmatrix}
			3 & 51/11 \\
			-2 & -3
		\end{pmatrix}$}  &\gape{$\displaystyle \frac{1}{2}(-2+\sqrt{3})$}\\
		\hline
		26+26& \gape{$\displaystyle \mathrm{i}\sqrt{\frac{3}{26}}$} &\gape{$\begin{pmatrix}
				0 & -1 \\
				26 & 0
			\end{pmatrix}\tau_0=\begin{pmatrix}
				1 & 0 \\
				0 & 3
			\end{pmatrix}\tau_0$} &\gape{$\begin{pmatrix}
			0 & 3/26 \\
			-1 & 0
		\end{pmatrix}$}  &\gape{$\displaystyle \hspace{0.2mm}\frac{1}{2}(11-3\sqrt{13})$\hspace{0.2mm}}\\
		\hline
		35+35& \gape{$\displaystyle \mathrm{i}\sqrt{\frac{2}{35}}$} &\gape{$\begin{pmatrix}
				0 & -1 \\
				35 & 0
			\end{pmatrix}\tau_0=\begin{pmatrix}
				1 & 0 \\
				0 & 2
			\end{pmatrix}\tau_0$} &\gape{$\begin{pmatrix}
			0 & 2/35 \\
			-1 & 0
		\end{pmatrix}$}  &\gape{$\displaystyle \sqrt{10}-3$}\\
		\hline
		39+39& \gape{$\displaystyle \mathrm{i}\sqrt{\frac{2}{39}}$} &\gape{$\begin{pmatrix}
				0 & -1 \\
				39 & 0
			\end{pmatrix}\tau_0=\begin{pmatrix}
				1 & 0 \\
				0 & 2
			\end{pmatrix}\tau_0$} & \gape{$\begin{pmatrix}
			0& 2/39 \\
			-1 & 0
		\end{pmatrix}$} &\gape{$3-2\sqrt{2}$}\\
		\hline
	\end{tabular}
\end{table}

\begin{table}[htbp]
	\caption{The coefficients $B,C$ of $1/\pi$ series.\label{BCFORSERIES}}
	\begin{tabular}{|c|c|c|}
		\hline
		\gape{$\Gamma$} & \gape{$B$} & \gape{$C$}\\
		\hline
		14+7& \gape{$\displaystyle \frac{3}{2}\sqrt{4-\frac{3 \sqrt{7}}{2}}$} &\gape{$\displaystyle \frac{1}{14}\left(588-223 \sqrt{7}+13 \sqrt{889-336 \sqrt{7}}\right) \sqrt{4+\frac{3 \sqrt{7}}{2}}$} \\
		\hline
		14+14& \gape{$\displaystyle 8 \sqrt{\frac{6}{527+115 \sqrt{21}}}$} &\gape{$\displaystyle \frac{4(747+163 \sqrt{21}) \sqrt{\frac{2}{3}(527-115 \sqrt{21})}}{(23+5 \sqrt{21})^{2}}$}\\
		\hline
		15+15& \gape{$\displaystyle 2 \sqrt{6(99-70 \sqrt{2})}$}& \gape{$\displaystyle 6 \sqrt{3(99-70 \sqrt{2})}+2(-536+379 \sqrt{2})\sqrt{3(99+70 \sqrt{2})}$}\\
		\hline
		16+& \gape{$\displaystyle 2(-2+\sqrt{6})\sqrt{15-6 \sqrt{6}}$} &\gape{$\displaystyle 2 (-12+5 \sqrt{6})\sqrt{\frac{1}{3}(5-2 \sqrt{6})}$}\\
		\hline
		20+20& \gape{\tabincell{c}{$\displaystyle -16(-2+\sqrt{3})$\\$\displaystyle \ \cdot\ \sqrt{3(97-56 \sqrt{3})}$}}&\gape{\tabincell{c}{$\displaystyle 4\left(14 \sqrt{97-56 \sqrt{3}}-7 \sqrt{3(97-56 \sqrt{3})}\right.$\\$\displaystyle \left.+3(-3064+1769 \sqrt{3})) \sqrt{97+56 \sqrt{3}}\right)$}} \\
		\hline
		21+21&\gape{\tabincell{c}{$\displaystyle 4(-2+\sqrt{6})$\\$\displaystyle \ \cdot \ \sqrt{98-40\sqrt{6}}$}}&\gape{\tabincell{c}{$\displaystyle \frac{2}{3}\left(-26 \sqrt{147-60 \sqrt{6}}+39 \sqrt{98-40 \sqrt{6}}\right.$\\$\displaystyle\left. +(-7035 \sqrt{2}+5744 \sqrt{3}) \sqrt{49+20 \sqrt{6}}\right)$}}\\
		\hline
		22+11 & \gape{$\displaystyle \sqrt{39-\frac{45 \sqrt{3}}{2}}$}&\gape{$\displaystyle \frac{1}{4}\left(7 \sqrt{52-30 \sqrt{3}}+3(-149+86 \sqrt{3}) \sqrt{52+30 \sqrt{3}}\right)$} \\
		\hline
		26+26&\gape{$\displaystyle 12 \sqrt{-8574+2378 \sqrt{13}}$} &\gape{ $\displaystyle 2(-41828+11601 \sqrt{13}) \sqrt{8574+2378 \sqrt{13}}$}\\
		\hline
		35+35&\gape{$\displaystyle 2\sqrt{14(721-228\sqrt{10})}$}&\gape{$\displaystyle 2(2\sqrt{2}-\sqrt{5})\ \sqrt{7(721-228\sqrt{10})}$}\\
		\hline
		39+39& \gape{\tabincell{c}{$\displaystyle -4(-2+\sqrt{2})$\\$\displaystyle \cdot\sqrt{6(577-408 \sqrt{2})}$}}& \gape{\tabincell{c}{$\displaystyle 44 \sqrt{3(577-408 \sqrt{2})}-22 \sqrt{6(577-408 \sqrt{2})}$\\$\displaystyle +2(-47420+33531\sqrt{2})\sqrt{3(577+408\sqrt{2})}$}}\\
		\hline
	\end{tabular}
\end{table}

\clearpage

\medskip

\section*{Acknowledgements} We would like to express our sincere thanks to professor T. Huber for providing us with preprint \cite{TDDN}.




\enlargethispage{3mm}

\end{document}